\documentclass[12pt, a4paper, reqno]{amsart}
\usepackage[a4paper,top=2cm,left=1.8cm,right=1.8cm]{geometry}

\usepackage{lmodern}
\usepackage{microtype}

\usepackage{amsmath}        
\usepackage{amsfonts}       
\usepackage{amsthm}         
\usepackage{amssymb}%
\usepackage{bm}             
\usepackage{graphicx}       
\usepackage[dvipsnames]{xcolor}  
\usepackage[colorlinks=true]{hyperref}
\hypersetup{citecolor=Sepia, linkcolor=Sepia, urlcolor=MidnightBlue}
\usepackage[T1]{fontenc} 
\usepackage{mathtools}
\usepackage[shortlabels]{enumitem}
\setlist[enumerate,1]{label=(\alph*), font=\normalfont}
\setlist[enumerate,2]{label=(\roman*), font=\normalfont,, ref=\theenumi(\roman*)}

\usepackage{csquotes}
\usepackage{bbding}

\theoremstyle{plain}
\newtheorem{theorem}{Theorem}
\newtheorem{lemma}[theorem]{Lemma}
\newtheorem{proposition}[theorem]{Proposition}
\newtheorem{corollary}[theorem]{Corollary}
\newtheorem{conjecture}[theorem]{Conjecture}
\newtheorem*{KitaokaConj}{Kitaoka's conjecture}

\theoremstyle{definition}

\DeclareMathOperator{\disc}{disc}

\DeclareMathOperator{\Tr}{Tr}
\DeclareMathOperator{\NN}{N} 

\newcommand{\R}{\mathbb{R}}
\newcommand{\Z}{\mathbb{Z}}
\newcommand{\Q}{\mathbb{Q}}
\newcommand{\Complex}{\mathbb{C}}

\newcommand{\ii}{\mathrm{i}}
\newcommand{\ee}{\mathrm{e}}

\newcommand{\vct}[1]{\mathbf{#1}} 

\newcommand{\vu}{\vct{u}}
\newcommand{\vv}{\vct{v}}

\newcommand{\ve}{\varepsilon}
\newcommand{\abs}[1]{\mathopen|#1\mathclose|}

\newcommand{\co}{\mathcal{O}}
\newcommand{\U}{\mathcal{U}}
\newcommand{\UK}[1][K]{\U_{#1}}
\newcommand{\UKPlus}[1][K]{\UK[{#1}]^+}
\newcommand{\UKctv}[1][K]{\UK[{#1}]^2}
\newcommand{\OK}[1][K]{\co_{{\:\!\! #1}}}
\newcommand{\OKPlus}[1][K]{\OK[#1]^+}

\newcommand{\UTL}{universal ternary lattice}
\newcommand{\UTF}{universal ternary form}

\def\house#1{{%
		\setbox0=\hbox{$#1$}
		\vrule height \dimexpr\ht0+1.4pt width .4pt depth \dp0\relax
		\vrule height \dimexpr\ht0+1.4pt width \dimexpr\wd0+2pt depth \dimexpr-\ht0-1pt\relax
		\llap{$#1$\kern1pt}
		\vrule height \dimexpr\ht0+1.4pt width .4pt depth \dp0\relax
}}

\newcommand{\qq}[1]{\Q(\!\sqrt{#1})} 
\newcommand{\arXiv}[1]{\href{https://arxiv.org/abs/#1}{arXiv:#1}}

\title[Kitaoka's conjecture for real cyclotomic fields]{Kitaoka's conjecture for real cyclotomic fields}

\author{Jakub Kr\'asensk\'y}
\address{Czech Technical University in Prague, Faculty of Information Technology, Department of Applied Mathematics, Thákurova~9, 16000 Praha~6, Czech Republic}
\email{jakub.krasensky@fit.cvut.cz}

\keywords{universal quadratic form, quadratic lattice, ternary quadratic form, cyclotomic field, real cyclotomic field, totally real, Kitaoka's Conjecture}
\subjclass[2020]{11E12, 11E20, 11R04, 11R18, 11R80}

\begin{document}
	\begin{abstract}
		We prove Kitaoka's conjecture for all real cyclotomic fields, i.e.\ for fields of the form $\Q(\zeta_k^{}+\zeta_k^{-1})$ where $\zeta_k$ is a primitive $k$th root of unity. More precisely, we show that such a field admits a universal classical ternary quadratic form if and only if it is a real quadratic field.
	\end{abstract}
	
	\maketitle
	
	\section{Introduction}
	
	Many authors studied which integral quadratic forms are \emph{universal}, i.e.\ represent all natural numbers. The most famous example is the sum of four squares; its universality is precisely the Four-square theorem of Lagrange \cite{Lagrange}. Since it is well known that no positive definite quadratic form in three variables over $\Z$ can be universal, the universal forms in four variables received particular attention. For example, Liouville \cite{Liouville} proved that there are exactly seven universal forms of the type $x^2+ay^2+bz^2+abw^2$; from Pépin's work \cite{Pépin} follows the universality of several other quaternary quadratic forms. Ramanujan \cite{Ramanujan} then completely classified which forms $ax^2+by^2+cz^2+dw^2$ are universal, making the famous mistake of including the form with coefficients $(1,2,5,5)$, which in fact represents all natural numbers except $15$. This was pointed out by Dickson \cite{Dickson}, who also systematically studied universality of non-diagonal forms (i.e.\ with cross-terms) and introduced the word \emph{universal} itself. Willerding's work \cite{Willerding} contains a classification of quaternary universal quadratic forms that are \emph{classical}, i.e.\ all their cross-terms are even. Much later, universal classical quadratic forms of any rank have been elegantly classified by the 15-Theorem of Conway--Schneeberger \cite{Conway-15} (a simpler proof has been provided by Bhargava \cite{Bhargava-15}); this also allowed to check Willerding's list and correct mistakes therein. The generalisation to non-classical quadratic forms is the 290-Theorem by Bhargava--Hanke, which has a very complicated proof that has not yet been fully published.
	
	Thus, universal quadratic forms over $\Q$ are now fully understood. However, there is a parallel story over number fields. We first introduce the necessary terminology. Let $K$ be a totally real number field, i.e.\ such that all embeddings $\sigma: K \hookrightarrow \Complex$ satisfy $\sigma(K) \subset \R$. We use $\OK$ for its ring of integers and $\OKPlus = \{\alpha \in \OK \mid \sigma(\alpha) > 0 \text{ for all embeddings $\sigma$}\}$ for the set of all \emph{totally positive} integers. A \emph{quadratic form} over $K$ is a homogeneous polynomial $Q$ of degree $2$, i.e.\ $Q(x_1,\ldots, x_n) = \sum_{1\leq i\leq j \leq n} a_{ij}x_ix_j$ with $a_{ij} \in \OK$. Such a form is \emph{classical} if $a_{ij} \in 2\OK$ for all $i\neq j$, and \emph{totally positive definite} if $Q(\vct{x})\in\OKPlus$ for all $\vct{x} \neq \vct{0}$. In this paper, all quadratic forms are assumed to be totally positive definite. A form in three (four) variables is \emph{ternary} (\emph{quaternary}). A form is \emph{universal} if $Q(\OK^{\;\!n}) = \OKPlus \cup\{0\}$, or in other words if it represents every totally positive integer.
	
	We pointed out that a universal quadratic form over $\Q$ needs at least four variables; this holds more generally for totally real number fields of odd degree \cite{EK}. Over fields of even degree there is no such restriction, and in fact, Maass \cite{Maass} proved that the sum of three squares over $\qq{5}$ is universal. On the other hand, Siegel \cite{Siegel} showed shortly afterwards that $\Q$ and $\qq{5}$ are the only totally real fields where the sum of any number of squares is universal. Every totally real number field admits a universal quadratic form; one can construct it using the asymptotic local--global principle of Hsia--Kitaoka--Kneser \cite{HKK}. However, such a construction yields no bound on the number of variables, and, indeed, there are many papers exhibiting families of number fields over which universal forms require arbitrarily many variables \cite{BK1, BK2, Ka1, Ka3, KS, KT, Man, Ya}. Of particular interest is the paper \cite{KYZ} that shows that for any chosen $N$, quadratic fields admitting a universal $N$-ary quadratic form have density zero among all quadratic fields. On the other hand, Kim \cite{Kim-octonary} showed there are infinitely many quadratic fields admitting a universal form in $8$ variables.
	
	Just as quaternary universal forms over $\Q$, ternary universal forms over number fields received significant attention. Kitaoka conjectured that they are very rare:
	
	\begin{KitaokaConj}
		There are only finitely many totally real number fields that admit a universal ternary classical quadratic form.
	\end{KitaokaConj}
	
	The only three fields known to admit a universal ternary classical quadratic form are $\qq{2}$, $\qq{3}$ and $\qq{5}$, found in \cite{CKR}. These three quadratic fields can also be expressed as $F_k:= \Q(\zeta_k)^+:=\Q(\zeta_k^{}+\zeta_k^{-1})$ for $k=8, 12$ and $5$, respectively. In this short note, we are going to prove that no other field of this form (a \enquote{real cyclotomic field}) admits a classical universal ternary quadratic form:
	
	
	\begin{theorem} \label{th:main}
		The field $\Q(\zeta_k)^+$ admits a classical universal ternary quadratic form if and only if it is one of $\qq{2}$, $\qq{3}$ and $\qq{5}$.
	\end{theorem}
	
	Since the only known fields exceptional with respect to Kitaoka's conjecture are real cyclotomic fields, we consider this theorem quite a significant piece of evidence for the validity of Kitaoka's conjecture. Note that Theorem \ref{th:main} for the case when $k$ is a prime or a power of $2$ was proved in \cite[Thm.~1.4]{KY-EvenBetter}.
	
	Let us give an overview of results that directly concern Kitaoka's conjecture. We already mentioned that by Earnest--Khosravani \cite{EK}, a universal ternary form can only exist in fields of even degree. Chan--Kim--Raghavan \cite{CKR} fully characterised quadratic fields admitting a classical universal ternary form by proving that the only such fields are $\qq{2}$, $\qq{3}$ and $\qq{5}$. Kim--Kim--Park \cite{KKP} recently showed that there are only finitely many quadratic fields admitting a universal form in seven variables, which of course raises the question whether Kitaoka's conjecture is not in fact valid with \enquote{ternary} replaced by \enquote{septenary}. Nevertheless, even the original Kitaoka's conjecture continues to defy proof. The case of degree four has been studied by K.--Tinková--Zemková in \cite{KTZ} for biquadratic fields and then fully solved in the recent preprint Kramer--K. \cite{KK}; it turns out that there are no quartic fields admitting a classical universal ternary form. Based on this and other evidence, the authors ventured to conjecture a strengthened version of Kitaoka's conjecture:
	
	\begin{conjecture}[{\cite[Conj.~1.2]{KK}}] \label{co:KitaokaStrong}
		Let $K$ be a totally real number field such that there exists a classical universal ternary quadratic form
		. Then $K=\Q(\!\sqrt2)$, $\Q(\!\sqrt3)$, or $\Q(\!\sqrt5)$. 
	\end{conjecture}
	
	Apart from evidence from low-degree fields, the only general results are the \enquote{Weak Kitaoka Theorem} by Kala--Yatsyna \cite{KY-weak} that Kitaoka's conjecture holds for fields of every fixed degree, and the recent paper \cite{KKK} by Kala--Kramer--K. which connects Kitaoka's conjecture with sums of squares by showing that if $K \not\ni \sqrt2$ admits a classical universal ternary form, then every element of $2\OKPlus$ is represented by $x^2+y^2+z^2+w^2$. This is a very useful result that allowed to confirm the Explicit Kitaoka's Conjecture \ref{co:KitaokaStrong} for fields with odd discriminant \cite[Thm.~1.1]{KKK} and played a crucial role in proving the main result of \cite{KK} (the validity of Conjecture \ref{co:KitaokaStrong} for all quartic fields).
	
	The main result of the present paper (Theorem \ref{th:main}) and its earlier special case \cite[Thm.~1.4]{KY-EvenBetter} are unusual as they explicitly resolve Kitaoka's conjecture for a (very natural) family of number fields of unbounded degree. Note that even Man's in-depth study of multiquadratic fields \cite{Man}, which essentially shows that the number of variables of a universal form over a multiquadratic field grows with the discriminant of the field in each fixed degree, allows for a universal ternary form to exist in every degree -- Kitaoka's conjecture for multiquadratic fields is still open.
	
	One can also hope to use the present result for an attack on Kitaoka's conjecture for a general totally real number field $K$: On the one hand, if $K$ is abelian, then it is contained in some cyclotomic field $\Q(\zeta_k)$ and hence also in its maximal totally real subfield $F_k=\Q(\zeta_k)^+$. On the other hand, and this is probably the more promising direction, it is natural to classify $K$ based on the real cyclotomic fields $F_k$ it contains, and to try to extend the argument for the non-existence of a universal ternary form over a maximal subfield $F_k$ to the same argument over $K$. We succeed with this idea in the case when $K$ contains a nonsquare totally positive unit, see Theorem \ref{th:general}.
	
	\section{Preliminaries}
	
	We already gave part of the necessary definitions at the beginning of the paper (in the second paragraph). Let us now continue with the rest.
	
	First of all, we already defined a quadratic form, but many of the modern results are formulated in a more general framework of quadratic lattices. A \emph{quadratic lattice} (or simply a lattice) is a pair $(L,Q)$ where $L$ is a finitely generated torsion-free $\OK$-module equipped with a \emph{quadratic map} $Q: L \to \OK$, i.e.\ a map such that $Q(\alpha\vu)=\alpha^2Q(\vu)$ and that the associated symmetric map $B_Q(\vu,\vv) = \frac12 ( Q(\vu+\vv)-Q(\vu)-Q(\vv))$ is bilinear. It is \emph{classical} if the image of $B_Q$ is contained in $\OK$, \emph{totally positive definite} if $Q(\vu) \in \OKPlus$ for all nonzero $\vu\in L$, and \emph{universal} if $Q(L) = \OKPlus \cup \{0\}$. The \emph{rank} of the lattice is the rank of the corresponding module, i.e.\ the dimension of the vector space $K \otimes L$; in particular, a lattice is \emph{ternary} is this rank is $3$.
	
	Every quadratic form $Q$ can be interpreted as a free quadratic lattice $(\OK^{\;\!n},Q)$, and the notions of classicality, universality, total positive definiteness and rank carry over. This also explains why classical quadratic forms are so important -- they correspond to integral symmetric bilinear forms. Conversely, every free quadratic lattice gives rise to a quadratic form simply by choosing a basis. As a trivial corollary, a field that admits a universal ternary form also admits a universal ternary lattice. Hence, Kitaoka's conjecture is sometimes formulated in a stronger version for lattices (cf. the original wording of Conjecture \ref{co:KitaokaStrong} in \cite[Conj.~1.2]{KK}). A significant part of our results handles this more general case (see Theorem \ref{th:general}), but sometimes we rely on results that have not been proven for non-free lattices.

	We use $\U_K$ for the group of units in $K$, and put $\UKPlus=\U_K \cap \OKPlus$ and $\UKctv=\{\eta^2 \mid \eta \in \UK\}$. Clearly, $\UKctv \subset \UKPlus$, and the size of the factorgroup $\UKPlus/\UKctv$ plays a significant role in the study of universal quadratic forms. If $\abs{\UKPlus/\UKctv} > 2$, then Kitaoka's conjecture trivially holds:
	
	\begin{proposition}[{\cite[Prop.~2.5]{KKK}}] \label{pr:KKK2.5}
		Let $K$ be a totally real number field. Assume that there exists a universal ternary lattice over $K$. Then $\abs{\UKPlus/\UKctv}\leq2$.
	\end{proposition}
	
	Thus, the study of Kitaoka's conjecture subdivides into two cases: Fields where $\UKPlus=\UKctv$, and fields with a totally positive unit that is not a square; by the above proposition we may assume that all such units differ only by multiplication by a square. This allows us to use \cite[Thm.~3.1]{KKK}, which will be one of our most important tools. 
	
	
	
	\section{Real cyclotomic fields}
	
	Let us now introduce our subject of study, the real cyclotomic fields. We write $\zeta_k = \ee^{2\pi \ii /k}$ for the primitive $k$th root of unity. Then $\Q(\zeta_k)$ is the \emph{$k$th cyclotomic field}; its degree is computed by the Euler function: $[\Q(\zeta_k):\Q] = \varphi(k)$. Its maximal totally real subfield is $\Q(\zeta_k)^+ := \Q(\zeta_k+\zeta_k^{-1})$; we usually just denote it $F_k$, and call it a \emph{real cyclotomic field}. In fact, if $k\geq 3$, then $[\Q(\zeta_k):F_k]=2$, which means that $\Q(\zeta_k)$ is a so-called CM-field.
	
	\begin{lemma}\label{le:ramification}
		Let $p\in\Z^+$ be a prime. If $p$ ramifies in $F_k$, then $p \mid k$. 
	\end{lemma}
	\begin{proof}
		If $k=1$ or $k=2$, then $F_k=\Q$ and the statement is trivial. Assume $k\geq 3$. The well-known discriminant formula for cyclotomic fields \cite[Prop.~2.7]{Washington}
		\[
		\disc \Q(\zeta_k) = (-1)^{\varphi(k)/2} \frac{k^{\varphi(k)}} {\prod\limits_{p \mid k} p^{\varphi(k)/(p-1)}}.
		\]
		implies that $p$ ramifies in $\Q(\zeta_k)$ if and only if $p$ divides $k$. Since $\Q \subset F_k \subset \Q(\zeta_k)$, a prime ramified in $F_k$ is ramified in $\Q(\zeta_k)$, which yields the result.
	\end{proof}
	
	We have $F_1=F_2=F_3=F_4=F_6=\Q$ and $F_5 = \qq{5}$, $F_8=\qq{2}$ and $F_{12}=\qq{3}$. For all other values, $[F_k:\Q] = \varphi(k)/2 > 2$, which allows to reformulate Theorem \ref{th:main} as \enquote{$F_k$ admits a classical universal ternary form if and only if it is a quadratic field}. The cubic field $F_7$ with discriminant $49$ occurs as an exceptional field in many different contexts, in particular in connection with the lifting problem for quadratic forms, see \cite{GT, KKL, KL, KY-lifting, KY-weak, KY-EvenBetter}, but also with class number of sums of squares \cite{Pe2} and therefore with Pythagoras number \cite{Kr}. The quartic field $F_{20} = \Q\bigl(\!\sqrt{(5+\smash[t]{\sqrt5})/2} \mskip2mu\bigr)$ is one of the very few fields where sums of squares can be characterised by local conditions \cite{Sch-thesis, Sch}; all the six mentioned $F_k$ are among the only eight known fields where all of $2\OKPlus$ are sums of squares \cite{KY-EvenBetter}.
	
	Real cyclotomic fields do not share all the nice properties of cyclotomic fields -- for example, the composite of two real cyclotomic fields is not necessarily real cyclotomic; e.g., the biquadratic field $\Q(\!\sqrt{2},\sqrt{5})$ is not. However, we still have the inclusions one would expect:
	
	\begin{lemma} \label{le:inclusions}
		The fields $F_k = \Q(\zeta_k)^+$ have the following properties:
		\begin{enumerate}
			\item If $m \mid k$, then $F_m \subset F_k$. \label{it:inclusion}
			\item $F_{2k}=F_k$ if and only if $k$ is odd (or $k=2$). \label{it:2k=k}
		\end{enumerate}
	\end{lemma}
	\begin{proof}
		We have $\Q(\zeta_m) \subset \Q(\zeta_k)$, so the same property holds for their intersections with $\R$. For the second part, we already know $F_k \subset F_{2k}$, so the fields are equal if and only if they have the same degree, i.e.\ if $\varphi(2k)=\varphi(k)$ (or if $k\leq2$). 
	\end{proof}
	
	Universal quadratic forms are intimately connected with indecomposable algebraic integers. We call $\alpha\in\OKPlus$ \emph{indecomposable} if it cannot be written as a sum of two elements of $\OKPlus$. Another of our crucial tools is the following result that provides some explicit indecomposable elements in $F_k$.
	
	\begin{proposition}[{\cite[Prop.~3.9]{KK}}] \label{pr:KKKcyclo}
		For $k\geq3$, define
		\[
		\alpha_k=2+\zeta_k^{}+\zeta_k^{-1} \quad \text{and} \quad \beta_k=2-\zeta_k^{}-\zeta_k^{-1}.
		\]
		Then the following hold:
		\begin{enumerate}
			\item $\alpha_k,\beta_k\in\OKPlus[F_k]$; \label{pr:KKKcyclo-totpos}
			\item $\NN_{F_k/\Q}(\alpha_k)=\begin{cases}
				p & \text{if } k=2p^n \text{ (where $p$ is a prime and $n\geq1$),}\\
				1 & \text{otherwise;}
			\end{cases} 
			$
			
			\vspace{2pt}\noindent $\NN_{F_k/\Q}(\beta_k)=\begin{cases}
				p & \text{if } k=p^n \text{ (where $p$ is a prime and $n\geq1$),}\\
				1 & \text{otherwise;}
			\end{cases}
			$
			\label{pr:KKKcyclo-norm}
			\vspace{2pt}
			\item $\Tr_{F_k/\Q}(\alpha_k)=2[F_k:\Q]+\mu(k)$ and $\Tr_{F_k/\Q}(\beta_k)=2[F_k:\Q]-\mu(k)$, where $\mu(k)$ is the M\"obius function. \label{pr:KKKcyclo-trace}
			\item If $[F_k:\Q]>2$, then both $\alpha_k$ and $\beta_k$ are indecomposable in $F_k$. \label{pr:KKKcyclo-indec}
		\end{enumerate}
	\end{proposition}
	
	Obviously $F_k=\Q(\alpha_k)=\Q(\beta_k)$, hence by Lemma \ref{le:inclusions}, $F_k$ contains all $\alpha_i$ and $\beta_i$ for $i \mid k$. That might seem to produce many indecomposable elements in $F_k$; however, the statement only claims indecomposability in $F_i$, not in its extensions. Let us also remark that $\alpha_k$ and $\beta_k$ are intimately connected: Depending on $k$ (more specifically its $2$-valuation), one easily sees that $-\zeta_k$ is either a $2k$th, or a $(k/2)$th, or a $k$th primitive root of unity; it follows that if $4 \mid k$, then $\alpha_k$ and $\beta_k$ are conjugates, while if $k$ is odd, then $\alpha_k$ is a conjugate of $\beta_{2k}$ and vice versa.
	
	It will be important for us to see when $\alpha_k$ is a square. The answer is contained in the following lemma.
	
	\begin{lemma}\label{le:whensquare}
		Let $k\geq 3$ and consider $\alpha_k=2+\zeta_k^{}+\zeta_k^{-1}$. Then:
		\begin{enumerate}
			\item $\qq{\alpha_k}=F_{2k}$; \label{it:sqrt}
			\item $\alpha_k$ is a square in $F_k$ if and only if $k$ is odd. \label{it:whensquare}
		\end{enumerate}
	\end{lemma}
	\begin{proof}
		Directly from the definition we get
		\[
		\alpha_k = 2 + {\ee}^{2\pi \ii/k} + {\ee}^{-2\pi \ii/k} = ({\ee}^{\pi \ii/k} + {\ee}^{-\pi \ii/k} )^2 = (\zeta_{2k}^{}+\zeta_{2k}^{-1})^2.
		\]
		This yields $\sqrt{\alpha_k}=\pm(\zeta_{2k}+\zeta_{2k}^{-1})$, which is a generator of $F_{2k}$. This proves the first part.
		
		In particular, $\alpha_k$ is a square in $F_k$ if and only if $F_k$ contains $F_{2k}$; by Lemma \ref{le:inclusions}\ref{it:2k=k}, this happens if and only if $k$ is odd (in which case these two fields are identical). This proves the second part.
	\end{proof}
	
	
	We shall not need it, but it is interesting to note that the identity from the proof can also be written as $\alpha_k = 2(1+\cos\frac{2\pi}{k}) = ( 2\cos\frac{\pi}{k})^2$. For us, the crucial corollary of Lemma \ref{le:whensquare} is the following:
	
	\begin{lemma}\label{le:nonsquare}
		\begin{enumerate}
			\item If $k>2$ is even and not of the form $k=2p^n$ for a prime $p$, then $F_k$ contains a nonsquare totally positive unit, namely $\alpha_k=2+\zeta_k+\zeta_k^{-1}$. \label{it:nonsquare-a}
			\item If $F$ is a real cyclotomic field where $\UKPlus[F] = \UKctv[F]$, then $F = F_{p^n}$ for a prime $p$ and $n\in\Z^+$. \label{it:nonsquare-b}
			\item In particular, if $k\neq 2^n$ and $4 \mid k$, then $\UKPlus[F_k]\neq\UKctv[F_k]$. \label{it:nonsquare-c}
		\end{enumerate}
	\end{lemma}
	\begin{proof}
		Total positivity of $\alpha_k$ for $k \geq 3$ is contained in Proposition \ref{pr:KKKcyclo}\ref{pr:KKKcyclo-totpos}. By Proposition \ref{pr:KKKcyclo}\ref{pr:KKKcyclo-norm} and Lemma \ref{le:whensquare}\ref{it:whensquare}, we see that if $k$ is even and not of the form $2p^n$, then $\alpha_k$ is a nonsquare unit in $F_k$. This proves \ref{it:nonsquare-a}. Part \ref{it:nonsquare-c} is an immediate trivial corollary.
		
		To prove \ref{it:nonsquare-b}, assume $\UKPlus[F] = \UKctv[F]$ and write $F=F_k$; thanks to the identity $F_{k}=F_{2k}$ for odd $k$ (Lemma \ref{le:inclusions}\ref{it:2k=k}), we can assume that $k$ is even. Also, $F_1=F_2=F_4$ allows us to assume $k > 2$. Hence, part \ref{it:nonsquare-a} yields $k=2p^n$ for a prime $p$. Thus either $F=F_k=F_{2^{n+1}}$, or $p$ is odd and $F=F_{2p^n}=F_{p^n}$.
	\end{proof}
	
	We remark that, somewhat annoyingly, part \ref{it:nonsquare-b} can not be formulated simply as \enquote{if $k\neq p^n$, then $\UKPlus[F_k]\neq\UKctv[F_k]$}, since one must also exclude the case $k=2p^n$. One can also note that it is not an equivalence; for instance, $F_{29}$ contains nonsquare totally positive units.

	\section{The proof}
	
	After this preparation, we are able to start the proof of our main result.
	
	\begin{lemma}\label{le:eps}
		Let $K$ be a totally real number field where $\UKPlus \neq \UKctv$. Assume $F_{p^n} \subset K$ for a prime $p$ and $n \in \Z^+$, where
		\begin{enumerate}
			\item $p\geq 7$, or \label{it:>7}
			\item $p \in \{3,5\}$, $n\geq 2$, or \label{it:35}
			\item $p=2$, $n \geq 3$. \label{it:2}
		\end{enumerate}
		Then $K$ admits no classical \UTL.
	\end{lemma}
	\begin{proof}
		If $\abs{\UKPlus / \UKctv} > 2$, then $K$ trivially admits no classical \UTL{}, see Proposition~\ref{pr:KKK2.5}. Thus we can assume $\abs{\UKPlus / \UKctv} = 2$. This allows us to use \cite[Thm.~3.1]{KKK}.
		
		If condition \ref{it:2} is satisfied, then $\qq{2}=F_8 \subset K$, and thus $\sqrt2 \in K$. Hence $K$ admits no classical \UTL{} by \cite[Thm.~3.1(f)]{KKK}.
		
		If condition \ref{it:>7} or \ref{it:35} is satisfied, then we are going to exhibit $\beta \in \OKPlus$ with $\NN_{K/\Q}(\beta) < 2^{[K:\Q]}$ for which $\NN_{K/\Q}(\beta)$ is not a power of two. If this is achieved, then $K$ admits no classical \UTL{} by \cite[Thm.~3.1(h)]{KKK}.
		
		In both cases, we put $\beta = \beta_{p^n}$ as in Proposition \ref{pr:KKKcyclo}. It is a totally positive integer by Proposition \ref{pr:KKKcyclo}\ref{pr:KKKcyclo-totpos}. Using the formula $\NN_{K/\Q}(\alpha)=\NN_{F_{p^n}/\Q}(\alpha)^{[K:F_{p^n}]}$, valid for every $\alpha\in F_{p^n}$, we easily see that it is enough to prove that $\NN_{F_{p^n}/\Q}(\beta)$ is not a power of $2$ and that it is $<2^{[F_{p^n}:\Q]}$. We know
		\[
		\NN_{F_{p^n}/\Q}(\beta) = p
		\]
		by Proposition \ref{pr:KKKcyclo}\ref{pr:KKKcyclo-norm}. Thus it is clearly not a power of $2$. As for the inequality, we have
		\[
		[F_{p^n}:\Q]=\frac{\varphi(p^n)}{2} = p^{n-1}\frac{p-1}2,
		\]
		and one easily verifies that indeed $p < 2^{(p-1)/2}$ if $p\geq 7$, and $p < 2^{p(p-1)/2}$ if $p > 2$. This concludes the proof.
	\end{proof}
	
	We are able to strengthen this lemma by considering all real cyclotomic subfields $F_k$, not only those where $k$ is a prime power.
	
	\begin{theorem}\label{th:general}
		Let $K$ be a totally real number field where $\UKPlus \neq \UKctv$. Assume $F_k \subset K$ for some $F_k$ other than $\Q$, $\qq{3}$ and $\qq{5}$. Then $K$ does not admit a classical \UTL.
	\end{theorem}
	\begin{proof}
		If $m \mid k$, then $F_m \subset F_k$ by Lemma \ref{le:inclusions}. Therefore, if $K$ has $\UKPlus \neq \UKctv$ and admits a classical \UTL, then in light of Lemma \ref{le:eps}, $K$ can only contain $F_k$ if $k=2^a3^b5^c$ where $0\leq a \leq 2$ and $0 \leq b,c \leq 1$. If $k=1,2,3,4,6$ then $F_k = \Q$, for $k=5,10$ we have $F_k=\qq{5}$ and $k=12$ gives $F_k=\qq{3}$. These are the three exceptions listed in the statement of Theorem \ref{th:general}. Thus, to conclude the proof, it remains to handle $K \supset F_k$ for $k = 15, 20, 30$ and $60$. 
		
		The field $F_{20}$ is quartic, $\abs{\UKPlus[F_{20}]/\UKctv[F_{20}]}=2$, and it contains a totally positive element of norm $5<2^4$. (One such element is $\vartheta(\vartheta-1)$ where $\vartheta=\zeta_{20}+\zeta_{20}^{-1}$.) Thus, by \cite[Cor.~3.13]{KKK}, no extension of $F_{20}$ admits a classical \UTL. Note that this also covers extensions of $F_{60}$. This leaves us with $F_{15}=F_{30}$.
		
		Let $K$ contain $F_{30}$. By Lemma \ref{le:nonsquare}\ref{it:nonsquare-a}, we know that $F_{30}$ contains a nonsquare totally positive unit $\ve=\alpha_{30}$. If both $\ve$ and $2\ve$ are nonsquare in $K$, then $K$ admits no classical \UTL{} by \cite[Thm.~3.1(e)]{KKK} if $\abs{\UKPlus/\UKctv}=2$, or trivially by Proposition \ref{pr:KKK2.5} if $\abs{\UKPlus/\UKctv} > 2$. This means that $K$ contains either $\qq{\alpha_{30}}$ or $\qq{2\alpha_{30}}$.
		
		Lemma \ref{le:nonsquare}\ref{it:sqrt} yields $\qq{\alpha_{30}}=F_{60}$ for the former field, and fields containing $F_{60}$ were already handled. We claim that the latter field, $\qq{2\alpha_{30}}$, contains $\sqrt6$: One can readily check that $\alpha_{30}$ is a root of $x^4-8x^3+14x^2-7x+1$, so $\sqrt{2\alpha_{30}}$ is a root of $x^8-2\cdot 8x^6+4\cdot 14x^4-8\cdot 7x^2+16=x^8-16x^6+56x^4-56x^2+16$; and since this is an irreducible polynomial, it is its minimal polynomial. We can look this field up in the LMFDB database \cite{LMFDB}; it has the label 8.8.5184000000.2 and we indeed see that it contains not only $F_{15}$, but also $\qq{6}$. This concludes the proof, since no extension of $\qq{6}$ admits a classical \UTL{} by \cite[Cor.~3.13]{KKK}.
	\end{proof}
	
	The preceding theorem easily yields a large part of our main Theorem \ref{th:main} (see the corollary below), but it is also interesting on its own -- it shows that if a field admits a classical \UTL, then it cannot contain a nonsquare totally positive unit together with a large real cyclotomic subfield. It is another step (after \cite[Thm.~3.1]{KKK}) towards proving Kitaoka's conjecture in the case $\UKPlus \neq \UKctv$.
	
	\begin{corollary}\label{co:eps}
		If $F_k$ admits a classical \UTL, then either $\UKPlus[F_k] = \UKctv[F_k]$ or $F_k=\qq{3}$.
	\end{corollary}
	\begin{proof}
		This is Theorem \ref{th:general} applied with $K=F_k$; one just needs to note that $\Q$ fails to admit a \UTL{} and that $\UKPlus[\qq{5}] = \UKctv[\qq{5}]$.
	\end{proof}
	
	We are now ready to prove the main theorem. It is interesting to note that the proof considers three separate types of $F_k$ based on the value of $k$, and in each of these types, there is exactly one exceptional field that does admit a classical \UTF.
	
	\begin{proof}[Proof of Theorem \ref{th:main}]
		First of all, the fields $\qq{2}$, $\qq{3}$ and $\qq{5}$ indeed admit a classical \UTF{} by \cite{CKR}. The main task is to prove the other implication.  
		
		Let $F_k = \Q(\zeta_k)^+$. We will distinguish several cases. First, if $k$ is a power of $2$, then $F_k$ admits a classical \UTF{} if and only if $F_k=\qq{2}$ by \cite[Thm.~1.4]{KY-EvenBetter}.
		
		If $k$ is not a power of $2$ and $4 \mid k$, then by Lemma \ref{le:nonsquare}\ref{it:nonsquare-c}, $F_k$ contains a nonsquare totally positive unit. Thus, by Corollary \ref{co:eps}, it can only admit a classical \UTL{} if $F_k=\qq{3}$.
		
		Finally, if $4 \nmid k$, then we can assume that $k$ is odd, since $F_k = F_{2k}$ for odd $k$ by Lemma \ref{le:inclusions}. As $2$ is not a divisor of $k$, the prime $2$ is unramified in $F_k$ by Lemma \ref{le:ramification}. Thus $F_k$ admits no classical \UTF{} unless $F_k=\qq{5}$ by \cite[Thm.~1.1]{KKK}.
	\end{proof}
	
	We conclude this section by a strong corollary of Theorem \ref{th:general}.
	
	\begin{corollary}
		Let $K$ be a totally real number field. If $K$ contains a real cyclotomic field which is neither $F_{12}=\qq{3}$, nor of the form $F_{p^n}$ for a prime $p$, then $K$ does not admit a classical universal ternary lattice.
	\end{corollary}
	\begin{proof}
		Let the real cyclotomic field be $F$. Lemma \ref{le:nonsquare}\ref{it:nonsquare-b} yields $\UKPlus[F] \neq \UKctv[F]$. This proves $\UKPlus[K] \neq \UKctv[K]$, since the quantity $\abs{\UKPlus[{}]/\UKctv[{}]}$ cannot decrease under field extension of totally real number fields (see \cite[Rem.~1]{DDK}). Further, the assumptions imply $F\neq F_3=\Q$ and $F \neq F_5 = \qq{5}$. Thus $K$ satisfies all assumptions of Theorem \ref{th:general}, and hence admits no classical \UTL{}.
	\end{proof}
	
	Let us note that Theorem \ref{th:main} does not hold for non-classical forms: Indeed, the quartic field $F_{20}$ 
	does admit a non-classical universal ternary form $x^2 + y^2 + \ve z^2 + xy + \pi yz$, where we write $\vartheta=\zeta_{20}+\zeta_{20}^{-1}$, $\ve=\alpha_{20}=\vartheta+2$ and $\pi=\vartheta^3+\vartheta^2-3\vartheta-2$. Its universality was proven and shall be published in an upcoming paper by the present author and R. Scharlau.

\end{document}